\documentclass[11pt, a4paper]{amsart}
\usepackage{a4wide}
\usepackage{mathrsfs, amssymb,amsmath,url,amsthm}
\usepackage{tikz} 
\usetikzlibrary{positioning,decorations,matrix,arrows,through,fit}

\title{Uniform Birkhoff}

\author
{Mai Gehrke}
\address{IRIF, University Paris-Diderot and CNRS, France.}
\email{mgehrke@liafa.univ-paris-diderot.fr}
 \urladdr{http://www.irif.univ-paris-diderot.fr/~mgehrke/}
\thanks{The research of Mai Gehrke has received funding from the European Research Council (ERC) under the European Union's Horizon 2020 research and innovation programme (grant agreement No 670624). 
}
\author
{Michael Pinsker}
	\address{Department of Algebra, MFF UK, Sokolovska 83, 186 00 Praha 8, Czech Republic}    
	\email{marula@gmx.at}
    \urladdr{http://dmg.tuwien.ac.at/pinsker/}
\thanks{The research of Michael Pinsker has been funded through project P27600 of the  Austrian Science Fund (FWF)}

\theoremstyle{plain}

    \newtheorem{thm}{Theorem}[section]

    \newtheorem{prop}[thm]{Proposition}

    \newtheorem{cor}[thm]{Corollary}

\theoremstyle{definition}

    \newtheorem{defn}[thm]{Definition}

\theoremstyle{remark}

\DeclareMathOperator{\I}{I}
\DeclareMathOperator{\Clo}{Clo}

\newcommand{\ignore}[1]{}

\newcommand{\To}{\rightarrow}

\newcommand{\aA}{{\mathfrak A}}
\newcommand{\aB}{{\mathfrak B}}
\newcommand{\aC}{{\mathfrak C}}
\newcommand{\aF}{\ensuremath{\mathfrak{F}}}
\newcommand{\aT}{{\mathfrak T}}

\newcommand{\F}{{\mathscr F}}
\newcommand{\sT}{{\mathscr T}}

\newcommand{\cB}{\ensuremath{\mathcal{B}}}
\newcommand{\cC}{\ensuremath{\mathcal{C}}}
\newcommand{\cK}{\ensuremath{\mathcal{K}}}
\newcommand{\cU}{\ensuremath{\mathcal{U}}}
\newcommand{\cP}{\ensuremath{\mathcal{P}}}
\newcommand{\cV}{\ensuremath{\mathcal{V}}}
\newcommand{\cVf}{\ensuremath{\mathcal{V}_{\it {f}{g}}}}

\DeclareMathOperator{\con}{Con}
\DeclareMathOperator{\HH}{H}
\DeclareMathOperator{\Pf}{P_{\it fin}}
\DeclareMathOperator{\SPf}{SP^{\it fg}_{\it fin}}
\DeclareMathOperator{\Sf}{S_{\it fg}}

\newcommand{\ot}{\ensuremath{\overline{t}}}
\newcommand{\oa}{\ensuremath{\overline{a}}}

\begin{document}
\maketitle
\begin{abstract}
We show that pseudovarieties of finitely generated algebras, i.e., classes $\cC$ of finitely generated algebras closed under finite products, homomorphic images, and subalgebras, can be described via a uniform structure $\cU^\cC$ on the free algebra for $\cC$: the members of $\cC$ then are precisely those finitely generated algebras $\aA$ for which the natural mapping from the free algebra onto the term clone of $\aA$ is well-defined and uniformly continuous with respect to the uniformity $\cU^\cC$ and the uniformity of pointwise convergence on the term clone of $\aA$, respectively. Our result unifies earlier theorems describing pseudovarieties of finite algebras and the pseudovariety generated by a single oligomorphic algebra.
\end{abstract}

\section{Introduction}

The goal of the present paper is the unification of two topological generalizations of Birkhoff's HSP theorem involving finite, rather than arbitrary, products of algebras. Let $\tau$ be a functional signature. The generalizations concern the following two formulations of Birkhoff's theorem from~\cite{Bir-On-the-structure}; each of the two variants can easily be derived from the other one.

\subsection*{Birkhoff's theorem, globally and locally} The first, ``global" formulation states that a class $\cC$ of $\tau$-algebras is a \emph{variety}, i.e., closed under arbitrary products, subalgebras, and homomorphic images,  if and only if it is the set of all models of some set of $\tau$-identities. In that case there exists a countably generated free object in $\cC$, called the \emph{free algebra} for $\cC$, whose elements can be represented as the abstract $\tau$-terms over a countable set of generators factored by the $\tau$-identities defining $\cC$. The members of $\cC$ then are precisely those $\tau$-algebras for which every finitely generated subalgebra is a factor of the free algebra up to isomorphism, or put differently again, those $\tau$-algebras $\aB$ for which the natural assignment from the term clone $\F^\cC$ of the free algebra for $\cC$ onto the term clone $\Clo(\aB)$ of $\aB$, which sends every term function in $\F^\cC$ to the corresponding term function in $\Clo(\aB)$, is a well-defined mapping. In that case, this mapping is in fact a \emph{clone homomorphism}, i.e., it preserves composition and the projection mappings, cf.~\cite{Topo-Birk, Reconstruction, BPP-projective-homomorphisms}.

The second, ``local" formulation of Birkhoff's theorem concerns the variety \emph{generated} by a single $\tau$-algebra $\aA$, i.e., the class of all algebras that can be obtained from $\aA$ using arbitrary products (in that case, in fact, powers), subalgebras, and homomorphic images. It states that a $\tau$-algebra $\aB$ is a member of that variety 
 if and only if it satisfies all equations that hold in $\aA$, or in other words, if the natural clone homomorphism from the term clone $\Clo(\aA)$ of $\aA$ onto the term clone $\Clo(\aB)$ of $\aB$ which sends every $\tau$-term $t^\aA$ over $\aA$ to the corresponding $\tau$-term $t^\aB$ over $\aB$ is well-defined.

\subsection*{Topology I: pseudovarieties of finite algebras} In a series of three articles in the 1980s, the problem of finding a variant of the global  formulation for \emph{pseudovarieties of finite $\tau$-algebras}, i.e., classes of \emph{finite} algebras closed under \emph{finite} products, subalgebras, and homomorphic images, was addressed: first in~\cite{EilenbergSchuetzenberger} for monoids, then for arbitrary countable signatures in~\cite{Reitermann}, and finally in full generality in~\cite{Banaschewski}. It is known that pseudovarieties of finite $\tau$-algebras can, in general, not be defined by equations of $\tau$-terms; a simple example is the pseudovariety of finite nilpotent semigroups~\cite{Reitermann}. It turns out, however, that such classes $\cC$ can be defined by a certain uniform structure on the term clone $\F^\cC$  of the free algebra for $\cC$: the members of $\cC$ then are precisely those finite $\tau$-algebras $\aB$ for which the natural homomorphism from $\F^\cC$ onto the clone $\Clo(\aB)$ of $\aB$ is a well-defined uniformly continuous mapping.

\subsection*{Topology II: the pseudovariety generated by an oligomorphic algebra} Concerning the local formulation of Birkhoff's theorem, a characterization of the pseudovariety generated by an \emph{oligomorphic} $\tau$-algebra $\aA$, i.e., a countable algebra whose unary term operations contain an oligomorphic permutation group (see~\cite{Oligo}), has been obtained recently; this was motivated by applications to constraint satisfaction problems and connections with model theory~\cite{Topo-Birk}. Namely, an oligomorphic $\tau$-algebra $\aB$ is contained in that pseudovariety if and only if the natural clone homomorphism from $\Clo(\aA)$ onto $\Clo(\aB)$ is well-defined and extends continuously to the closures of the two clones with respect to the topology of pointwise convergence on functions (in~\cite{Topo-Birk}, this is expressed via Cauchy-continuity). Again, it was known that purely algebraic notions are not sufficient for such a characterization, unless both $\aA$ and $\aB$ are finite (and hence, with hindsight, carry the discrete topology on their term clones, so that the topological structure does not matter). In the latter case, it already follows from the proof of Birkhoff's theorem~\cite{Bir-On-the-structure} that $\aB$ is contained in the pseudovariety generated by $\aA$ if and only if it is contained in the variety generated by $\aA$, i.e., if the natural clone homomorphism from $\Clo(\aA)$ onto $\Clo(\aB)$ is well-defined. The theorem from~\cite{Topo-Birk} can thus be seen as a generalization from finite to oligomorphic algebras. In this context it is important to remark that all members of the pseudovariety generated by an oligomorphic algebra are again oligomorphic (whereas the members of the variety generated by such an algebra do not share this property in general).

\subsection*{The unification} Given these two topological theorems, the question whether they could be united to obtain a characterization of pseudovarieties of oligomorphic algebras arises naturally to the curious mind. In other words, is there a global variant of the local theorem from~\cite{Topo-Birk} characterizing the pseudovariety generated by a single oligomorphic algebra? Put yet differently, is there a variant of the global theorem from~\cite{EilenbergSchuetzenberger, Reitermann, Banaschewski} for oligomorphic, rather than finite algebras? Figure~\ref{fig:situation} illustrates this situation. We remark that the topologies of the local and the global theorem are unrelated: in the global characterization of pseudovarieties of finite algebras, the topology on the functions of the algebras of a pseudovariety is actually discrete, but an appropriate topology which reflects the pseudovariety as a whole is imposed only on the clone of the free algebra; on the other hand, in the local characterization of the pseudovariety generated by a single oligomorphic algebra, a particular topology, namely the topology of pointwise convergence, is considered on the functions of every single algebra.
\begin{figure}
\begin{center}
\begin{tabular}{|l | l | l | l|}
\hline
&Varieties&Pseudovarieties of finite&Pseudovarietes of oligomorphic\\\hline
local& \cite{Bir-On-the-structure} & \cite{Bir-On-the-structure} & \cite{Topo-Birk} \\\hline
global& \cite{Bir-On-the-structure} & \cite{EilenbergSchuetzenberger, Reitermann, Banaschewski} &{\large \textbf{?}}\\\hline
\end{tabular}
\caption{The situation}\label{fig:situation}
\end{center}
\end{figure}

We show that it is indeed possible to characterize pseudovarieties of oligomorphic, and even more generally, of finitely generated algebras topologically and algebraically via a synthesis of the results from~\cite{EilenbergSchuetzenberger, Reitermann, Banaschewski} and the one from~\cite{Topo-Birk}. By a \emph{pseudovariety of finitely generated algebras} we mean a class of finitely generated algebras containing all finitely generated algebras that can be obtained by closing the class under finite products, homomorphic images, and subalgebras.

\begin{thm}\label{thm:main}
Let $\cC$ be a pseudovariety of finitely generated $\tau$-algebras. Then there exists a uniformity $\cU^\cC$ on the clone $\F^\cC$ of the free algebra for $\cC$ such that the members of $\cC$ are precisely those finitely generated $\tau$-algebras $\aA$ for which the natural clone homomorphism from $\F^\cC$ onto $\Clo(\aA)$ is well-defined and uniformly continuous with respect to $\cU^\cC$ and the uniformity of pointwise convergence on $\Clo(\aA)$.
\end{thm}

Let us remark that if $\cC$ is taken to consist of finite algebras, then the results of~\cite{EilenbergSchuetzenberger, Reitermann, Banaschewski} follow; if on the other hand one takes $\cC$ to be generated by a single oligomorphic algebra, then the result from~\cite{Topo-Birk} is a consequence. Actually, concerning the latter, we obtain a more general theorem than the one in~\cite{Topo-Birk}, namely a description of the pseudovariety generated by a finitely generated algebra, rather  than an oligomorphic one (oligomorphic algebras are finitely generated); this variant has already proven useful in the context of \emph{h1 clone homomorphisms} and constraint satisfaction problems~\cite{wonderland, BKOPP}. Birkhoff's theorem, in both the global and the local variant presented here, follows from the proof of Theorem~\ref{thm:main} by disregarding the continuity condition, for finitely generated algebras. We refer to the discussion at the end of the present article, in Section~\ref{sect:discussion}, for more details.%

\section{Notation}

We fix a functional signature $\tau$. We denote $\tau$-algebras by $\aA, \aB,\ldots$ and write $A,B,\ldots$ for their domains. We let $\aT$ and $T$ be, respectively, the $\tau$-algebra and the set of all abstract $\tau$-terms over a countable set of variables. Similarly, for $k\geq1$, we let $\aT_k$ and $T_k$ be, respectively, the algebra and the set of all abstract $\tau$-terms over the first $k$ variables $x_1,\ldots,x_k$. When $\aA$ is a $\tau$-algebra and $t\in T_k$, then we write $t^\aA$ for the $k$-ary term operation induced on $A$ by interpreting $t$ as a term function of $\aA$. Moreover, we set $\Clo(\aA):=\{t^\aA\; |\; t\in T_k,\, k\geq 1\}$ to be the set of all term operations of $\aA$. This set is a \emph{function clone}, i.e., it is a family of functions on a fixed set closed under composition and containing the projections, and so we will also refer to it as the \emph{term clone} of $\aA$. In particular, we denote the term clone of the algebra $\aT$ by $\sT$ and refer to it as the \emph{absolutely free term clone for the signature $\tau$}. 

When $\cC$ is a class of $\tau$-algebras, then we write $\aF^\cC$ for the algebra $\aT$ factored by the congruence relation given by the equations $\Sigma(\cC)\subseteq T^2$ that hold for all $\aA\in\cC$. Of course, $\aF^\cC$ can be viewed naturally as the countably generated free algebra of the variety generated by $\cC$. We denote the term clone of $\aF^\cC$ by $\F^\cC$ and refer to it as the  \emph{term clone} of $\cC$. Note that the term clone of an algebra $\aA$ is isomorphic to the term clone of the variety generated by $\aA$. 

For each $k\geq1$, the $k$-ary term functions of a term clone carry a natural $\tau$-algebra structure, and in the case of $\F^\cC$, the $k$-th component, $F^{\cC}_k$, may naturally be seen as the domain of $\aF^{\cC}_k$, the $k$-generated free algebra in the variety generated by $\cC$. Accordingly, we will speak of the congruences of a term clone, and when $\cB$ is a subfamily of such congruences and $k\geq 1$, then we denote by $\cB_k$ its $k$-th component, i.e., those congruences of the family which are congruences on the $k$-ary functions of the clone. In particular, by $\con(\F^\cC)$ we shall mean the indexed family $(\con(\aF^{\cC}_k))_{k\geq1}$ of $\tau$-congruences of $\aF^{\cC}_k$ for $k\geq1$.

\section{Congruences of term clones}\label{sect:congruences}

Let $\cV$ be a variety. As we have just remarked, the components of its term clone $\F^{\cV}$, viewed as algebras, are the finitely generated free algebras in $\cV$. Their homomorphic images are precisely \emph{the class of finitely generated members of $\cV$}, the class of which we denote by $\cVf$. In this section, we identify subclasses of $\cVf$ which are determined by families of congruences of $\F^{\cV}$.

Consider the following assignments between subclasses $\cK$ of $\cVf$ and subfamilies $\cB$ of $\con(\F^\cV)$:

\[
\cK\quad\mapsto\quad \cB^\cK=(\cB^{\cK}_{k})_{k\geq1},\\
\]
where
\[
\cB^{\cK}_{k}=\{\theta\in\con(\aF_k^\cV)\mid \aF_k^\cV/\theta\in\I(\cK)\}\\[2ex]
\]
and $\I(\cK)$ is the closure of $\cK$ under isomorphic images, and

\[
\cB\quad\mapsto\quad \cK^\cB= \I(\bigcup_{k\geq1}\{\aF^\cV_k/\theta\mid \theta\in\cB_k\}).\\
\]

An algebra in $\cVf$ may be obtained in various different ways as a quotient of a finitely generated free algebra of $\cV$, but as soon as we restrict ourselves to subclasses $\cK$ of $\cVf$ that are closed under finitely generated subalgebras, we obtain a nice correspondence with families of congruences of the clone $\F^\cV$.

\begin{defn}
For a subclass $\cK$ of $\cVf$, we denote by $\Sf(\cK)$ the class of all algebras which are isomorphic images of finitely generated subalgebras of members of $\cK$. We say that $\cK$ is \emph{closed under finitely generated subalgebras} provided that $\Sf(\cK)=\cK$.
\end{defn}

The following definition gives the corresponding property for families of congruences of $\F^\cV$.

\begin{defn}
Let $k,m\geq1$ and $\theta\in\con(\aF^\cV_k)$. Further, let $\ot=(t_1,\ldots,t_m)\in(\aF^\cV_k)^m$. We define $\theta/\ot$ to be the congruence on $\aF^\cV_m$ given by
\[
s\,(\theta/\ot)\, s' \quad \iff \quad  s(t_1,\ldots,t_m)\, \theta\,  s'(t_1,\ldots,t_m).\\
\]
We say that a subfamily $\cB$ of $\con(\F^\cV)$ is \emph{closed under inverse substitution} provided, for all $k\geq 1$, for all $\theta\in\cB_k$, for all $m\geq 1$, and for all $\ot\in(\aF^\cV_k)^m$,  we have $\theta/\ot\in\cB_m$.\\
\end{defn}

\begin{prop}\label{prop:basic1-1}
The assignments $\cK\ \mapsto\ \cB^\cK$ and $\cB\ \mapsto\ \cK^\cB$ between subclasses of $\cVf$ and subfamilies of $\con(\F^\cV)$ establish a one-to-one correspondence between the subclasses of $\cVf$ closed under finitely generated subalgebras and the subfamilies of $\con(\F^\cV)$ closed under inverse substitution.
\end{prop}

\begin{proof}
Let $\cK$ be a subclass of $\cVf$ which is closed under $\Sf$ and let $\theta\in\cB^\cK$. Then there is an $\aA\in\cK$ and a surjective homomorphism $h:\aF^\cV_k\to \aA$ with kernel $\theta$. Further, let $\ot=(t_1,\ldots,t_m)\in(\aF^\cV_k)^m$. Note that $\theta/\ot$ is the kernel of the composition
\begin{center}
\begin{tikzpicture}
\matrix (m) [matrix of math nodes, row sep=.25em, column sep=1.75em, text 
height=1.5ex, text depth=0.25ex] 
{  \aF^\cV_m & \aF^\cV_k & \aA \\
x_i\ &\ \ t_i&\\};
\path[|->] (m-2-1) edge node[above] {} (m-2-2);
\path[->] (m-1-1) edge node[above] {}  (m-1-2);
\path[->] (m-1-2) edge node[above] {$h$}  (m-1-3);
\end{tikzpicture}
\end{center} 
and that the image of this composition is a finitely generated subalgebra of $\aA$. Since $\cK$ is closed under finitely generated subalgebras, it follows that $\theta/\ot\in\cB^\cK$.

Similarly, let $\cB$ be subfamily of $\con(\F^\cV)$ closed under inverse substitution and suppose $\aA$ belongs to $\cK^\cB$. Then there exists a $k\geq 1$ and a surjective homomorphism $h\colon\aF^\cV_k\to \aA$ with $\ker(h)\in\cB_k$. Now let $e\colon \aB\hookrightarrow \aA$ be a finitely generated subalgebra of $\aA$ and let $b_1,\ldots,b_m$ be a generating set for $\aB$. For each $b_i$ pick $t_i\in\aF^\cV_k$ with $h(t_i)=e(b_i)$. Then
\begin{center}
\begin{tikzpicture}
\matrix (m) [matrix of math nodes, row sep=.25em, column sep=1.75em, text 
height=1.5ex, text depth=0.25ex] 
{  \aF^\cV_m & \aF^\cV_k & \aB \\
x_i\ &\ \ t_i&\ b_i\\};
\path[|->] (m-2-1) edge node[above] {} (m-2-2);
\path[|->] (m-2-2) edge node[above] {} (m-2-3);
\path[->] (m-1-1) edge node[above] {}  (m-1-2);
\path[->] (m-1-2) edge node[above] {$h$}  (m-1-3);
\end{tikzpicture}
\end{center} 
is a surjective homomorphism and its kernel is $\ker(h)/\overline{t}$. Since $\cB$ is closed under inverse substitution, it follows that $\aB\in\cK^\cB$ as required.

Finally, given any subclass $\cK$ of $\cVf$ which is closed under isomorphic copies, we have $\cK^{\cB^\cK}=\cK$, and for any subfamily $\cB$ of $\con(\F^\cV)$ and any $l\geq 1$ we have $\cB_l\subseteq\cB^{\cK^\cB}_l$.  Now suppose $\cB$ is closed under inverse substitution and let $\theta\in\cB^{\cK^\cB}_l$. Then there is $k\geq 1$ and $\psi\in\cB_k$ so that $\aF^\cV_l/\theta\cong\aF^\cV_k/\psi$. Let $\alpha\colon\aF^\cV_k/\psi\to\aF^\cV_l/\theta$ witness the isomorphism of these two algebras. Then there are $t_1\ldots,t_k\in\aF^\cV_k$ so that $\alpha([t_i]_\psi)=[x_i]_\theta$ for each $i\in\{1,\ldots,k\}$, and then the quotient map from $\aF^\cV_l$ to $\aF^\cV_l/\theta$ is equal to the composition
\begin{center}
\begin{tikzpicture}
\matrix (m) [matrix of math nodes, row sep=.25em, column sep=1.75em, text 
height=1.5ex, text depth=0.25ex] 
{  \aF^\cV_l & \aF^\cV_k & \aF^\cV_k/\psi&\aF^\cV_l/\theta \\
x_i\ &\ \ t_i\;.&\\};
\path[|->] (m-2-1) edge node[above] {} (m-2-2);
\path[->] (m-1-1) edge node[above] {}  (m-1-2);
\path[->] (m-1-2) edge node[above] {}  (m-1-3);
\path[->] (m-1-3) edge node[above] {$\alpha$}  (m-1-4);
\end{tikzpicture}
\end{center} 
We note that the kernel of this map is $\psi/\ot$ and is thus a member of $\cB_l$ as required.
\end{proof}

Note that finite products of finitely generated algebras are not necessarily finitely generated. For example, the algebra obtained by equipping $\mathbb N$ with the unary successor function is clearly finitely generated, but in the square of that algebra every element of the form $(0,n)$ needs to be in the generating set if it is to be in a subalgebra, since $0$ is not the successor of any natural number. 

For this reason we require an altered version of the operator $\Pf$ which closes classes of algebras under finite products as we only want to pick out the finitely generated members. 

\begin{defn}
For a subclass $\cK$ of $\cVf$ we denote by $\SPf(\cK)$ all isomorphic copies of finitely generated subalgebras of finite products of members of $\cK$. We say that $\cK$ is \emph{closed under finitely generated subalgebras of finite products} provided that $\SPf(\cK)=\cK$.

We denote by $\HH$ the operator on classes of algebras which yields the closure under homomorphic images. Note that $\cVf$ is closed under the operator $\HH$.

Finally, a class $\cC$ of finitely generated $\tau$-algebras will be called a \emph{pseudovariety of finitely generated algebras} provided it is closed under $\HH$ and $\SPf$. That is, $\cC$ contains precisely the finitely generated members of the pseudovariety generated by $\cC$.
\end{defn}

 We can now restrict the one-to-one correspondence of Proposition~\ref{prop:basic1-1} to classes closed under finitely generated subalgebras of finite products. 

\begin{prop}\label{prop:1-1products}
The assignments $\cK\ \mapsto\ \cB^\cK$ and $\cB\ \mapsto\ \cK^\cB$ establish a one-to-one correspondence between the subclasses of $\cVf$ closed under finitely generated subalgebras of finite products and the subfamilies of $\con(\aF)$ closed under inverse substitution and finite intersection.
\end{prop}

\begin{proof}
Let $\cK$ be a class closed under finitely generated subalgebras and let $\cB$ be the corresponding subfamily of  $\con(\aF)$, which is then closed under inverse substitution. We show that $\cK$ is closed under finitely generated subalgebras of finite products if and only if $\cB$ is closed under intersection.

Suppose $\cK$ is closed under finitely generated subalgebras of finite products and let $\theta_1,\ldots,\theta_n\in\cB_k$. Then $\bigcap_{j=1}^n\theta_j$ is the kernel of the product of the maps $h_i\colon\aF^{\cV}_k\to \aF^{\cV}_k/\theta_i$ and this product  map
\[
\times_{j=1}^n h_j\colon\aF^{\cV}_k\to \times_{j=1}^n \aF^{\cV}_k/\theta_j
\]
has as image a finitely generated subalgebra of a product of algebras each of which belongs to $\cK$. Thus the range of $\times_{j=1}^n h_j$ is an element of $\cK$ and consequently $\ker(\times_{j=1}^n h_j)=\bigcap_{j=1}^n\theta_j$ is in $\cB_k$.

Conversely, suppose $\cB$ is closed under finite intersections and suppose $\aA$ embeds via $e$ as a finitely generated subalgebra of a product $\aA_1\times\cdots\times \aA_n$ where each $\aA_j$ belongs to $\cK$. Let $a_1,\ldots,a_k$ be a set of generators for $\aA$. For each $j\in\{1,\ldots,n\}$, let $h_j$ be the composition 
\begin{center}
\begin{tikzpicture}
\matrix (m) [matrix of math nodes, row sep=.25em, column sep=1.75em, text 
height=1.5ex, text depth=0.25ex] 
{ \aF^{\cV}_k & \aA & \aA_1\times\cdots\times \aA_n & \aA_j\\
x_i\ &\ \ b_i\;.&&\\};
\path[->] (m-1-1) edge node[above] {}  (m-1-2);
\path[->] (m-1-2) edge node[above] {$e$}  (m-1-3);
\path[->] (m-1-3) edge node[above] {$\pi_j$}  (m-1-4);
\path[|->] (m-2-1) edge node[above] {} (m-2-2);
\end{tikzpicture}
\end{center} 
Then the image $\Im(h_j)$ of $h_j$ is a finitely generated subalgebra of $\aA_j$ and thus the kernel, $\theta_j$, of this composition belongs to $\cB_k$. Thus, by our assumption, $\bigcap_{j=1}^n\theta_j\in\cB_k$, and thus the image of the product map

\[
\aF^\cV_k\to \times_{j=1}^n \aF^\cV_k/\theta_j\cong \times_{j=1}^n \Im(h_j)\\[1ex]
\]
belongs to $\cK$. Finally note that the image of $x_i$ under this map is $(\pi_1(e(a_i)),\ldots,\pi_n(e(a_i)))=e(a_i)$ for each $i\in\{1,\ldots,k\}$ and thus the image of this map is $A$ and thus it follows that $\aA$ belongs to $\cK$ as required.
\end{proof}

Note that that closure of a class under homomorphic images corresponds to taking the upset of the corresponding subfamily of  $\con(\aF^\cV)$. We have the following corollary.

\begin{cor}\label{cor}
The subclasses of $\cVf$ that are pseudovarieties of finitely generated algebras are in one-to-one correspondence with the filters of $\con(\F^\cV)$ which are closed under inverse substitution.
\end{cor}

\section{Proof of the main theorem}

Theorem~\ref{thm:main} speaks of various uniformities on clones. We start by recalling here the notion of a uniformity (also see, for example,~\cite[Chapter~2]{Bourbaki}) and by introducing the required uniformities.

A \emph{uniform space} is a pair $(X,\cU)$ where $X$ is a set and $\cU$ is \emph{uniformity} on $X$, that is, $\cU\subseteq\cP(X\times X)$ satisfying:
\begin{enumerate}
\item   $\cU$ is a filter contained in the principal filter generated by the diagonal of $X$, that is, each element of $\cU$ contains $\Delta_X=\{(x,x)\mid x\in X\}$;
\item  $\cU$ is closed under taking converses, that is, $U\in\cU$ implies $U^{-1}=\{(y,x)\mid (x,y)\in U\}\in\cU$;
\item $\cU$ satisfies the following generalisation of the triangle inequality:
\[
\forall\ U\in\cU \ \ \exists V\in\cU\quad V\circ V\subseteq U.
\]
\end{enumerate}
The elements of a uniformity are called \emph{entourages}. Given uniform spaces $(X,\cU)$ and $(Y,\cV)$, a map $f\colon X\to Y$ is \emph{uniformly continuous} provided for each entourage $V$ of $Y$ the preimage
\[
\{(x,x')\in X\times X\mid (f(x),f(x'))\in V\}
\]
 is an entourage of $X$. When we wish to emphasize the uniformities in question, we also say that $f$ is \emph{$(\cU,\cV)$-uniformly continuous}.
 
 Note that if $X$ is a set and a subcollection $\cB$ of the principal filter generated by $\Delta_X$ already satisfies the above generalised triangle inequality, then the filter of $\cP(X\times X)$ generated by the closure under converse of $\cB$ is a uniformity. In this case we say that $\cB$ is a \emph{basis} for a uniformity on $X$ and that the filter generated by its closure under converse is the \emph{uniformity generated by $\cB$}. Note in particular that any set of equivalence relations on $X$ generates a uniformity on $X$ simply by taking the upset in $\cP(X\times X)$ of the closure under finite intersections of the set.

\begin{defn}
Let $\cV$ be the variety of $\tau$-algebras and let $\F^\cV$ denote the corresponding clone of term functions. 
Further let $\cC$ be a pseudovariety of finitely generated algebras in $\cV$ and let $\cB$ be the corresponding filter of  $\con(\F^\cV)$ as provided by Corollary~\ref{cor}. We define $\cU^\cC$ to be the uniformity on $\F^\cV$ given by the basis $\cB$. That is, $\cU^\cC$
consists, for each $k\geq 1$, of a uniformity $\cU^\cC_k$ on $F^\cV_k$ given by
\[
\cU_k^\cC=\{U\subseteq F^\cV_k\times F^\cV_k\mid \ \exists\theta\in\cB_k\ \ \theta\subseteq U\}.
\]

Given a finitely generated algebra $\aA$, we denote by $\cU^\aA$ the uniformity on $\Clo(\aA)$ corresponding 
to the pseudovariety of finitely generated algebras generated by $\aA$, where we identify $\Clo(\aA)$ naturally with the set $F^\cV$, i.e., the domain of the free algebra $\aF^\cV$ for the variety $\cV$ generated by $\aA$.
\end{defn}

The following proposition states that the uniformity $\cU^\aA$ is just the usual \emph{uniformity of pointwise convergence} on functions, as considered, in particular, in~\cite{Topo-Birk, Reconstruction}.

\begin{prop}\label{prop:pointwise}
Given a $\tau$-algebra $\aA$, the uniform structure $\cU^\aA$ on $\Clo(\aA)$ is the uniformity 
generated (as a filter), for each $k\geq 1$, by the sets of the form
$$U_{\oa}=\{(f,g)\in \Clo(\aA)^2\mid f,g \text{ are }k\text{-ary},\; f(\oa)=g(\oa)\}$$ for $\oa\in A^k$.
\end{prop}
\begin{proof}
Denote by $\cV$ the variety generated by $\aA$. 

For $\oa\in A^k$, we first show that  $U_{\oa}\in\cU^\aA$. To see this, note that $U_{\oa}$ induces a congruence $\theta$ on $\aF_k^\cV$ via the identification of $F_k^\cV$ with the $k$-ary functions of $\Clo(\aA)$. The factor algebra $\aF_k^\cV/\theta$ is isomorphic to the subalgebra of $\aA$ generated by the elements of $\oa$, via the natural extension of the mapping which sends the equivalence class in $F_k^\cV/\theta$ of the $i$-th $k$-ary projection to the $i$-th component of $\oa$. Hence, $\aF_k^\cV/\theta$ is contained in the pseudovariety of finitely generated algebras generated by $\aA$ and $U_{\oa}\in\cU^\aA$.

Conversely, let $\theta$ be a congruence of $\aF_k^\cV$ such that $\aF_k^\cV/\theta$ is contained in the pseudovariety of finitely generated subalgebras generated by $\aA$. So there exists $m\geq 1$ and tuples $\oa_1,\ldots,\oa_k\in A^m$ such that $\aF_k^\cV/\theta$ is a factor of the algebra $\aC$ generated by those tuples in $\aA^m$. Then $\theta$ contains the congruence $\theta'$ of $\aF_k^\cV$ such that $\aF_k^\cV/\theta'$ is isomorphic to $\aC$. In $\theta'$, two (classes of) terms are identified if and only if they agree in $\aA^m$ on the tuple $(\oa_1,\ldots,\oa_k)$, which is ensured if they agree on the $k$-tuples of the $i$-th components of the tuples $\oa_1,\ldots,\oa_k\in A^m$, for each $1\leq i\leq m$. Denote the $k$-tuples of those $i$-th components by $\oa^i$. Then $\theta$ contains, via the identification of $F_k^\cV$ with the $k$-ary functions of $\Clo(\aA)$, the set $U_{\oa^1}\cap\cdots\cap U_{\oa^m}$, and is thus contained in said filter.
\end{proof}

We now restate Theorem~\ref{thm:main} in more detail and provide its proof.

\begin{thm}[Uniform Birkhoff]
Let $\cC$ be a pseudovariety of finitely generated $\tau$-algebras, let $\cV$ be a variety containing it, and let $\aA$ be a finitely generated 
$\tau$-algebra. Then the following are equivalent 
\begin{enumerate}
\item $\aA\in\cC$;
\item The natural clone homomorphism from $\F^\cV$ onto $\Clo(\aA)$ is well-defined and $(\cU^\cC,\cU^{\aA})$-uniformly continuous.
\end{enumerate}
\end{thm}

\begin{proof}
Suppose $\aA\in\cC$. Then in particular $\aA\in\cV$, and thus the natural clone homomorphism $\xi\colon\F^\cV \to\Clo(\aA)$ is a well defined function. Because $\aA\in\cC$, the uniformity $\cU^\cC$ on $\F^\cV$ contains, in particular, the uniformity induced by $\cU^\aA$ on $\F^\cV$ via $\xi$, so that $\xi$ is indeed $(\cU^\cC,\cU^\aA)$-uniformly continuous.

For the converse, suppose that the natural clone homomorphism from $\aF^\cV$ onto $\Clo(\aA)$ exists and is 
$(\cU^\cC,\cU^\aA)$-uniformly continuous, and let $\{a_1,\ldots,a_k\}$ be a generating set for $\aA$. Setting $\oa:=(a_1,\ldots,a_k)$, by Proposition~\ref{prop:pointwise} we have that the preimage of $U_{\oa}$ under $\xi$ is an element of $\cU^\cC$ and hence contains a congruence $\theta$ of $\aF^\cV_k$ such that $\aB:=\aF^\cV_k/\theta\in\cC$. It follows that for $k$-ary $\tau$-terms $f,g$, if $f^\aB=g^\aB$, then $f^\aA(\oa)=g^\aA(\oa)$. Hence, the natural homomorphism $h\colon \aF^\cV_k\To \aA$ which sends every representative $f$ of a class to $f^\aA(\oa)$ factors through the analogous homomorphism from $\aF_k^\cV$ onto $\aB$. Hence, there exists a homomorphism $g\colon \aB\to \aA$
so that
\begin{center}
\begin{tikzpicture}
\matrix (m) [matrix of math nodes, row sep=1.75em, column sep=1.75em, text 
height=1.5ex, text depth=0.25ex] 
{  \aF^\cV_k & \aB \\
&\aA\\};
\path[->] (m-1-1) edge node[above] {} (m-1-2);
\path[->] (m-1-1) edge node[above] {}  (m-2-2);
\path[->] (m-1-2) edge node[right] {$g$}  (m-2-2);
\end{tikzpicture}
\end{center} 
commutes. Also, notice that the homomorphism from $\aF_k^\cV$ to $\aA$ is surjective since 
$\{a_1,\ldots,a_k\}$ is a generating set for $\aA$. It follows that $g$ is surjective and thus 
$\aA$ is in $\cC$.
\end{proof}

\section{Discussion}\label{sect:discussion}

As a corollary, we obtain the theorem from~\cite{EilenbergSchuetzenberger, Reitermann, Banaschewski} characterizing pseudovarieties of finite algebras globally. To this end, we only have to observe that by Proposition~\ref{prop:pointwise}, the uniformity $\cU^\aB$ on $\Clo(\aB)$ is discrete for finite $\aB$. 

\begin{cor}\label{cor:esrb}
Let $\cC$ be a pseudovariety of finite $\tau$-algebras, and let $\cV$ be a variety containing $\cC$. 
Then  the members of $\cC$ are precisely those finite $\tau$-algebras $\aB$ for which the natural mapping 
from $\F^\cV$ onto $\Clo(\aB)$ is well-defined and uniformly continuous with respect to $\cU^\cC$ and the discrete topology, respectively.
\end{cor}

A further observation is that, in the case of a pseudovariety of finite $\tau$-algebras, the uniformity on $\aF^\cC$ is a \emph{Pervin uniformity} (generated by `blocks' of the form $(S\times S)\cup(S^c\times S^c)$) and thus its uniform completion is a Stone space \cite{GGP10}  and thus compact. This allows a reformulation of the above corollary in purely point-set topological terms, see e.g. \cite{Almeida}.\\

Another consequence is the following generalization of the theorem in~\cite{Topo-Birk} for finitely generated algebras, rather than oligomorphic ones. In fact, Lemma~10 of that article proves precisely that in oligomorphic setting, continuity implies uniform continuity. Therefore, the following can also be observed directly from their proof, a fact the authors of~\cite{Topo-Birk} were, however, unaware of at the time.

\begin{cor}\label{cor:topo-birk}
Let $\aA, \aB$ be finitely generated $\tau$-algebras. Then $\aB$ is contained in the pseudovariety generated by $\aA$ if and only if the natural mapping from $\Clo(\aA)$ onto $\Clo(\aB)$ is well-defined and uniformly continuous with respect to pointwise convergence.
\end{cor}

\bibliographystyle{plain}
\bibliography{global.bib}

\end{document}